\documentclass[12pt,a4paper]{amsart}
\usepackage{graphicx,amsmath, amssymb, overpic}

\usepackage{color}
\input xy

\xyoption{all}

\newcommand{\ad}{\ensuremath{\operatorname{ad}}}
\newcommand{\Aut}{\ensuremath{\operatorname{Aut}}}

\newcommand{\SL}{\ensuremath{\operatorname{SL}}}

\newcommand{\PSL}{\ensuremath{\operatorname{PSL}}}
\newcommand{\PGL}{\ensuremath{\operatorname{PGL}}}

\newcommand{\Syl}{\ensuremath{\operatorname{\mathcal{S}yl}}}

\newcommand{\F}{\mathbb{F}}

\newcommand{\Z}{\mathbb{Z}}
\newcommand{\G}{\Gamma}
\newcommand{\s}{\sigma}
\newcommand{\bs}{\backslash}
\newcommand{\forget}[1]{}

\newcommand{\cC}{\mathcal{C}}

\newcommand{\cK}{\mathcal{K}}

\newcommand{\cS}{\mathcal{S}}

\newcommand{\cX}{\mathcal{X}}
\newcommand{\cY}{\mathcal{Y}}
\newcommand{\cZ}{\mathcal{Z}}
\newcommand{\la}{\langle}
\newcommand{\ra}{\rangle}

\newtheorem{theorem}{Theorem}
\newtheorem{prop}[theorem]{Proposition}
\newtheorem{lemma}[theorem]{Lemma}
\newtheorem{corollary}[theorem]{Corollary}
\newtheorem*{surface}{Surface Subgroup Corollary}

\theoremstyle{definition}

\newtheorem{defn}{Definition}

\newtheorem*{example}{Example}

\title[Cocompact lattices in Kac--Moody groups]{Cocompact lattices in complete Kac--Moody groups with Weyl group right-angled or a free product of spherical special subgroups} 

\author{Inna Capdeboscq} \address{Mathematics Institute, Zeeman Building, University of Warwick, Coventry
CV4 7AL, UK} \email{I.Korchagina@warwick.ac.uk}\thanks{The second author is supported in part by an Australian Postdoctoral Fellowship and ARC Grant No. DP110100440.}

\author{Anne Thomas}\address{School of Mathematics and Statistics, Carslaw Building F07, University of Sydney NSW 2006, Australia}\email{anne.thomas@sydney.edu.au}

\date{\today}

\begin{document}

\begin{abstract}  Let $G$ be a complete Kac--Moody group of rank $n \geq 2$ over the finite field of order $q$, with Weyl group $W$ and building $\Delta$.  We first show that if $W$ is right-angled, then for all $q \not \equiv 1 \pmod 4$ the group $G$ admits a cocompact lattice $\G$ which acts transitively on the chambers of $\Delta$.  We also obtain a cocompact lattice for $q \equiv 1 \pmod 4$ in the case that $\Delta$ is Bourdon's building.  As a corollary of our constructions, for certain right-angled~$W$ and certain $q$, the lattice $\G$ has a surface subgroup.  We also show that if~$W$ is a free product of spherical special subgroups, then for all $q$, the group $G$ admits a cocompact lattice $\G$ with $\G$ a finitely generated free group.  Our proofs use generalisations of our results in rank~$2$~\cite{CT} concerning the action of certain finite subgroups of $G$ on $\Delta$, together with covering theory for complexes of groups.  \end{abstract}

\maketitle

\section*{Introduction}

A complete Kac--Moody group $G$ is the completion, with respect to some topology, of a minimal Kac--Moody group over a finite field $\F_q$ of order $q$.  We use the completion in the ``building topology" (see~\cite[Section 1.2]{CR}).  Complete Kac--Moody groups are locally compact, totally disconnected topological groups, which act transitively on the chambers of their associated building $\Delta$.  For further background, see our earlier work~\cite{CT}, which considered complete Kac--Moody groups of rank $2$. 

For $q$ large enough, a complete Kac--Moody group $G$ over $\F_q$ admits non-cocompact lattices, as established by Carbone--Garland \cite{CG} and independently R\'emy \cite{R}.  In rank $n = 2$, where the building for $G$ is a tree, various constructions of cocompact lattices in $G$ are also known (see \cite{CT} and its references).  However if~$G$ is of rank $n \geq 3$ and its root system has an affine subsystem of type other than $\tilde{A}_d$, for $d \geq 1$, then $G$ does not admit any cocompact lattice (see \cite[Remark 4.4]{CM}).  Our main results, Theorem \ref{t:RA} and \ref{t:FP} below, give explicit constructions of cocompact lattices in certain complete Kac--Moody groups $G$ of rank $n \geq 2$ whose root systems have $\tilde{A}_1$ as their only type of affine root subsystem.   As we discuss further below the statements of these results, we believe that in many cases we have obtained the first constructions of cocompact lattices in $G$.

A complete Kac--Moody group $G$ is defined using a generalised Cartan matrix $A$.  Recall that $A = (a_{ij})$ is an $n \times n$ integer matrix, such that all diagonal entries $a_{ii}$ equal $2$, all off-diagonal entries $a_{ij}$ are $\leq 0$ and $a_{ij} = 0$ if and only if $a_{ji} = 0$.  
We denote by $W$ the Weyl group of $G$, generated by the set $S = \{ s_i \}_{i \in I}$ with $I = \{ 1,\ldots,n\}$.   Then $(W,I)$ is a Coxeter system  with $W$ having presentation 
$W = \la S \mid (s_is_j)^{m_{ij}}\ra$, 
where $m_{ii} = 1$ and for $i \neq j $ we have $m_{ij} = 2,3,4,6$ or $\infty$ as $a_{ij}a_{ji} = 0,1,2,3$ or $\geq 4$, respectively.  Recall that a \emph{special subgroup} of $W$ is a subgroup of the form $W_J:=\langle s_j \mid j \in J \rangle$ with $J \subset I$, and that if $W_J$ is a special subgroup then $(W_J,J)$ is a Coxeter system.  A special subgroup $W_J$ is \emph{spherical} if and only if it is finite.

For both of our theorems, we assume that if $a_{ij}a_{ji} \geq 4$ then $|a_{ij}|, |a_{ji}| \geq 2$.  This ensures that we can generalise results from \cite{CT} concerning the rank $1$ subgroups of $G$ and their action on the building for $G$.  We also assume without further remark that the Weyl group $W$ is infinite, as otherwise $G$ is compact.

In our first main result, Theorem \ref{t:RA} below, the assumptions on the generalised Cartan matrix $A$ imply that $m_{ij} \in \{ 2, \infty \}$ for $i \neq j$.  Hence the Weyl group $W$ of $G$ is a right-angled Coxeter group.    In part \eqref{i:p=2, q=3, RA} of Theorem \ref{t:RA}, this right-angled Coxeter group is arbitrary, while in part \eqref{i:q=1} we consider two special cases.  We note that in both \eqref{i:p=2, q=3, RA} and \eqref{i:q=1}, the rank $n$ of $G$ is unbounded.

\begin{theorem}\label{t:RA}  Let $G$ be a complete Kac--Moody group of rank $n \geq 2$ with generalised Cartan matrix $A = (a_{ij})$, defined over the finite field $\F_q$ where $q = p^h$ and $p$ is prime.
Assume that for all $i \neq j$, either $a_{ij}a_{ji} = 0$ or $a_{ij}a_{ji} \geq 4$, and that if $a_{ij}a_{ji} \geq 4$ then $|a_{ij}|, |a_{ji}| \geq 2$.  Let $W$ be the Weyl group for $G$.

\begin{enumerate}
\item\label{i:p=2, q=3, RA} If $p = 2$ or $q\equiv 3\pmod 4$, then $G$ admits a chamber-transitive cocompact lattice $\G$. 
\item\label{i:q=1} Suppose $q\equiv 1\pmod 4$. 
\begin{enumerate}
\item\label{i:free} If for all $i \neq j$ we have $|a_{ij}| \geq 2$, that is, if $W$ is isomorphic to the free product of $n$ copies of the cyclic group of order $2$, then $G$ admits a cocompact lattice $\G$ which has $2$ orbits of chambers and is transitive on each type of panel.
\item\label{i:Ipq} If $n \geq 5$ and 
\[W = W_n := \la s_1, \ldots, s_n \mid s_i^2 = 1, \quad (s_i s_{i+1})^2 = 1, \quad \forall i \in \Z/n\Z \ra,\]
then $G$ admits a cocompact lattice $\G$ which has $F$ orbits of chambers, for any positive integer $F$ which is a multiple of $8$.
\end{enumerate}\end{enumerate}
\end{theorem}

 It is notable that the infinite families of lattices in part \eqref{i:p=2, q=3, RA} of Theorem \ref{t:RA} are chamber-transitive, since for affine buildings of dimension $\geq 2$ there exist very few chamber-transitive lattices (see~\cite{KLT} and its references).  The lattices in part \eqref{i:p=2, q=3, RA} of Theorem \ref{t:RA} generalise many of the edge-transitive lattices in $G$ of rank~$2$ that we obtained in \cite[Theorem 1.1]{CT}.  When $q \equiv 1 \pmod 4$ we do not know whether there exists a cocompact lattice in all $G$ with right-angled Weyl group.  We discuss why these values of $q$ are more difficult in  Section \ref{s:discuss q=1}.

In part \eqref{i:Ipq} of Theorem \ref{t:RA} above, the Weyl group $W_n$ is the group generated by reflections in the sides of a right-angled hyperbolic $n$--gon.  As we explain in Section \ref{s:proof RA Ipq} below, the building $\Delta$ for $G$ may then be realised as the right-angled Fuchsian building known as Bourdon's building \cite{Bo}.  Thus our results show that if the building associated to $G$ is Bourdon's building, then for all $q$ the Kac--Moody group $G$ admits a cocompact lattice.  Lattices in the full automorphism group of Bourdon's building have been studied by several authors (see, for instance, \cite{FT} and its references).

We explain in Section \ref{s:surface} below that if the Weyl group $W$ is word-hyperbolic then any cocompact lattice $\G < G$ is word-hyperbolic.  An open question of Gromov asks whether every one-ended word-hyperbolic group contains a surface subgroup, that is, a subgroup isomorphic to the fundamental group of a compact orientable hyperbolic surface.  Using our constructions and results on graph products of groups, in Section \ref{s:surface} we provide an affirmative answer to Gromov's question for some Kac--Moody lattices, as follows. 

\begin{surface}\label{c:surface}
Let $G$ be as in Theorem \ref{t:RA}, with right-angled Weyl group $W$.  In the following cases, the cocompact lattice $\G < G$ constructed in Theorem \ref{t:RA} has a surface subgroup:
\begin{enumerate}
\item $p = 2$ and $W$ has a special subgroup isomorphic to $W_{n'}$ for some $n' \geq 5$; or 
\item $q \equiv 1 \pmod 4$ and $W = W_n$ for some $n \geq 5$.
\end{enumerate}
Moreover, if $W = W_n$ for $n \geq 6$ then every cocompact lattice in $G$ has a surface subgroup.
\end{surface}

Our second main result, Theorem \ref{t:FP} below, constructs cocompact lattices in all $G$ such that the associated Weyl group $W$ is a free product of spherical special subgroups.    Theorem~\ref{t:FP} provides an alternative construction of a cocompact lattice when $W$ is a free product of copies of $C_2$, as in \eqref{i:free} of Theorem \ref{t:RA} above, but also includes cases in which $m_{ij}$ may equal $3$, $4$ or $6$.  We note that unlike Theorem \ref{t:RA}, the statement of Theorem \ref{t:FP} does not depend upon the value of $q$.

\begin{theorem}\label{t:FP}  
Let $G$ be a complete Kac--Moody group of rank $n \geq 2$ with generalised Cartan matrix $A = (a_{ij})$, defined over the finite field $\F_q$ where $q = p^h$ and $p$ is prime.
Assume that for all $i \neq j$,  if $a_{ij}a_{ji} \geq 4$ then $|a_{ij}|, |a_{ji}| \geq 2$, and that the Weyl group $W$ of $G$ has a free product decomposition $$W = W_1 * W_2 * \cdots * W_N$$ where for $1 \leq k \leq N$, each $W_k$ is a spherical special subgroup.  
 Then for all $q$, the group $G$ admits a cocompact lattice $\G$, with $\G$ being a finitely generated free group.
\end{theorem}

\noindent We provide a more explicit description of the action of $\G$ as in Theorem \ref{t:FP} in Section \ref{s:proof FP} below. 

In rank $n \geq 3$, the only previous constructions of cocompact lattices in non-affine complete Kac--Moody groups $G$ that are known to us are as follows.  

\begin{itemize} 
\item R\'emy--Ronan \cite[Section 4.B]{RR} constructed cocompact chamber-transitive lattices in certain groups defined using a twin root datum and having arbitrary right-angled Weyl group.  In their construction, the finite ground fields were ``mixed" (that is, of distinct characteristics). It seems that their construction might work for Kac--Moody groups $G$ with right-angled Weyl group when $G$ is defined over $\F_{q}$ and $q = 2^h$. Perhaps under some additional assumptions on $G$, their construction can also be carried out  when $q$ is odd.  (Using the notation introduced in Section~\ref{s:finite} below, these additional assumptions should include $Z(M_i)\leq Z(G)$ and $L_i/Z(L_i)\cong \PGL_2(q)$ for all $1 \leq i \leq n$, and $[L_i/Z(L_i), L_j/Z(L_j)]=1$ for all $i\neq j$ such that $m_{ij}=2$.)  Under such assumptions, some of the lattices we obtain in Theorem \ref{t:RA}\eqref{i:p=2, q=3, RA} above might then be equivalent to the lattices that would be obtained via the construction of \cite[Section 4.B]{RR}. 

\item Carbone--Cobbs \cite[Lemma 21]{CC} constructed a cocompact chamber-transitive lattice in $G$ in the special case that $n = 3$, $p = q = 2$ and $|a_{ij}| \geq 2$ for $i \neq j$.  Their lattice is the same as the lattice $\G$ in this case of Theorem \ref{t:RA}\eqref{i:p=2, q=3, RA} above.\footnote{The claim in the statement of \cite[Lemma 21]{CC} that the embedding of $\G$ into $G$ is nondiscrete is incorrect, as we now explain.  The building for $G$ in this case may be realised as a hyperbolic building with ideal vertices.  The $G$--stabilisers of the ideal vertices are noncompact and so may contain infinite discrete subgroups, which invalidates the reason given for nondiscreteness in the proof of \cite[Lemma 21]{CC}.  In fact in \cite[Section 7]{CC}, the abstract group $\G$ is correctly embedded in $G$ as a cocompact lattice, using a construction which is equivalent to that in our proof of Theorem \ref{t:RA}\eqref{i:p=2, q=3, RA} in this case.}

\item Gramlich--Horn--M\"uhlherr \cite[Section 7.3]{GHM} showed that for certain complete Kac--Moody groups~$G$, the fixed point set $G_\theta$ of a quasi-flip $\theta$ is a lattice in $G$.  The lattice $G_\theta$ is sometimes cocompact and sometimes non-cocompact.  A fundamental domain for $G_\theta$ is not known.
\end{itemize}
Thus for many $G$ of rank $n \geq 3$, our results provide the first (explicit) constructions of cocompact lattices in $G$.  In particular, Theorem \ref{t:FP}  gives the first explicit examples of cocompact lattices in complete Kac--Moody groups whose Weyl group is not right-angled.

To prove Theorems \ref{t:RA} and \ref{t:FP} above, we consider the action of $G$ on the Davis geometric realisation $X$ of the building $\Delta$.  We recall the construction of $X$ in Section \ref{s:Davis} below.  The geometric realisation $X$ is a  simplicial complex on which $G$ acts cocompactly with compact vertex stabilisers.  These properties of the $G$--action allow us to employ the following characterisation of cocompact lattices in $G$: a subgroup $\Gamma \leq G$ is a cocompact lattice if and only if $\Gamma$ acts on $X$ cocompactly with finite vertex stabilisers (see \cite{BL}).  

In Section \ref{s:finite} below, we gather results concerning certain finite subgroups of $G$ and their actions on $X$, which we generalise from \cite{CT}.  In the proof of Theorem \ref{t:RA} in Section \ref{s:proof RA} below, these finite subgroups will appear as the $\G$--stabilisers of the vertices of $X$.   In the proof of Theorem \ref{t:FP}, in Section~\ref{s:proof FP}, the lattices $\G$ that we construct have trivial vertex stabilisers.

Our proofs also use covering theory for complexes of groups.  This highly technical theory was developed by Bridson--Haefliger \cite[Chapter III.$\cC$]{BH} (see also Lim--Thomas \cite{LT} and Section \ref{s:complexes of groups} below).  We construct each cocompact lattice $\G < G$ as the fundamental group of a finite complex of finite groups, denoted $\G(\cY)$, such that there is a covering of complexes of groups from $\G(\cY)$ to the canonical complex of groups $G(\cK)$ induced by the action of $G$ on $X$.  It then follows from covering theory that $\G$ embeds in $G$ as a cocompact lattice.  We also obtain from the theory of complexes of groups a description of the action of $\G$ on~$X$, and that $\G$ as in Theorem \ref{t:FP} is a finitely generated free group.  In many of our proofs, the construction of the complex of groups $\G(\cY)$ together with a covering of complexes of groups $\G(\cY) \to G(\cK)$ is quite delicate.  This is not surprising, since the Kac--Moody group $G$ is ``small" inside the full automorphism group $\Aut(X)$, and so care is required in order to show that the fundamental group of $\G(\cY)$ actually embeds in $G$ as a cocompact lattice.

 \subsection*{Acknowledgements} We thank the University of Warwick for travel support.

\section{Preliminaries}\label{s:preliminaries}

Throughout this section, $G$ is a complete Kac--Moody group of rank $n \geq 2$ with generalised Cartan matrix $A = (a_{ij})$, defined over the finite field $\F_q$ where $q = p^h$ and $p$ is prime, such that for all $i \neq j$, if $a_{ij}a_{ji} \geq 4$ then $|a_{ij}|, |a_{ji}| \geq 2$.  Let $W$ be the Weyl group for $G$.  In Section \ref{s:Davis}, we recall the construction of the Davis geometric realisation $X$ of the building $\Delta$ for $G$, and describe the action of $G$ on $X$.     Section \ref{s:complexes of groups} then sketches the theory of complexes of groups and their coverings.  

\subsection{The Davis realisation}\label{s:Davis}

We will follow Davis \cite{D}, and assume the basic theory of buildings.  In the proofs of our main results in Sections \ref{s:proof RA} and \ref{s:proof FP} below, we will give more explicit descriptions of the Davis realisation $X$ where possible.
   
Denote by~$\cS$ the set of spherical subsets of $I$, that is, the subsets $J \subseteq I$ such that the special subgroup $W_J := \la s_j \mid j \in J \ra$ is finite.  By convention, $\emptyset \in \cS$ with $W_\emptyset$ the trivial group.   

Let $L$ be the \emph{nerve} of $W$.  That is, $L$ is the simplicial complex with vertex set $S$, so that for each $J \subseteq I$, the vertices $\{ s_j \}_{j \in J}$ span a simplex in $L$ if and only if $J \in \cS$.  Let $L'$ be the barycentric subdivision of $L$.  For each nonempty $J \in \cS$, denote by $\sigma_J$ the barycentre of the simplex of $L$ corresponding to $J$.  Then the set of vertices of $L'$ is in bijection with the set of nonempty $J \in \cS$.  To simplify notation, we write $\sigma_i$ for the vertex $\sigma_{\{i\}} = s_i$.

The \emph{chamber} $K$ is the cone on the barycentric subdivision $L'$.  For each $i \in I$, the chamber $K$ has \emph{mirror} $K_i$ which is, by definition, the closure of the star of the vertex $\sigma_i$ in $L'$.  Denote by $\sigma_\emptyset$ the cone point of $K$.  Then the set of vertices of $K$ is in bijection with the set of all spherical subsets $\cS$.  We will say that the vertex $\sigma_J$ has \emph{type} $J$.

For each $x \in K$, define
\[S(x) := \{ s_i \in S \mid x \in K_i \}. \]
Now let $\cC$ be the set of chambers of the building $\Delta$ and let $\delta: \cC \times \cC \to W$ be the $W$--valued distance function.  The Davis realisation $X$ is given by 
\[X = \cC \times K / \sim\]
where $(c,x) \sim (c',x')$ if and only if $x = x'$ and $\delta(c,c') \in W_{S(x)}$.  We identify $K$ with the subcomplex $(c_0, K)$ of $X$, where $c_0$ is the standard chamber of $\Delta$.  Each image of the mirror $K_i \subset K$ in $X$ is called a \emph{panel (of type $i$)} in $X$.  The vertices of $X$ naturally inherit types $J \in \cS$ from the vertices of $K$.

For each $i \in I$ we define $P_i := B \sqcup Bs_iB$.  For each $J \in \cS$, put $P_J := \sqcup_{w \in W_J} BwB$, so that in particular $P_\emptyset = B$.  Note that $P_i = P_{\{ i \}}$ for each $i \in I$.  The action of $G$ on $\Delta$ then induces a type-preserving action of $G$ on $X$, with quotient the chamber $K$, such that for each $J \in \cS$, the stabiliser of the vertex $\sigma_J \in K$ is $P_J$.  The kernel of this $G$--action is the finite group $Z(G)$, the centre of $G$ \cite{CR}.

\subsection{Complexes of groups}\label{s:complexes of groups}

In this section we recall the theory of complexes of groups and their coverings that we will need, following \cite[Chapter III.$\cC$]{BH}.  The main result is Corollary \ref{c:coverings}, which gives a sufficient condition for the fundamental group of a complex of groups to embed as a cocompact lattice in $G$.

We will mostly construct complexes of groups over \emph{scwols}, that is, small categories without loops.  The exception is in Section \ref{s:proof RA Ipq}, where we use constructions of complexes of groups over certain polygonal complexes instead.  

\begin{defn}\label{d:scwol} A \emph{scwol} $\cY$ is the disjoint union of a set $V(\cY)$ of vertices and a set $E(\cY)$ of edges, with each edge $a$ oriented from its initial vertex $i(a)$ to its terminal vertex $t(a)$, such that $i(a) \not = t(a)$.  A pair of
edges $(a,b)$ is \emph{composable} if $i(a)=t(b)$, in which case there is a third edge $ab$, called
the \emph{composition} of $a$ and $b$, such that $i(ab)=i(b)$ and $t(ab)=t(a)$, and if $(a,b)$ and $(b,c)$ are composable then $(ab)c = a(bc)$ (associativity).  \end{defn}

\begin{defn}\label{d:morphism_scwols}  Let $\cY$ and $\cZ$ be scwols.  A \emph{nondegenerate morphism} $f:\cY \to \cZ$ is a map that sends $V(\cY)$ to $V(\cZ)$ and $E(\cY)$ to $E(\cZ)$, such that:
\begin{enumerate}
\item for each $a \in E(\cY)$, we have $i(f(a)) = f(i(a))$ and $t(f(a)) = f(t(a))$;
\item for each pair of composable edges $(a,b)$ in $\cY$, we have $f(ab) = f(a)f(b)$; and
\item for each vertex $\s \in V(\cY)$, the restriction of $f$ to the set of edges with initial vertex $\s$ is a bijection onto the set of edges of $\cZ$ with initial vertex $f(\s)$.
\end{enumerate}
\end{defn}

A \emph{morphism of scwols} $f:\cY \to \cZ$ is a functor from the category $\cY$ to the category $\cZ$ (see \cite[Section~III.$\mathcal{C}$.A.1]{BH}). An \emph{automorphism} of a scwol $\cX$ is a morphism from $\cX$ to $\cX$ that has an inverse.  See \cite[Chapter III.$\mathcal{C}$]{BH} for the definition of the \emph{fundamental group} of a scwol and of a \emph{group action} on a scwol.

Suppose now that $X$ is the Davis realisation of the building $\Delta$ for the Kac--Moody group $G$, as defined in Section~\ref{s:Davis} above.  Recall that each vertex $\sigma \in V(X)$ has type $J$ a spherical subset of $I$.  The edges $E(X)$ are then naturally oriented by inclusion of type.  That is, if an edge $a$ joins a vertex $\sigma'$ of type $J'$ to
a vertex $\sigma$ of type $J$, then $i(a)=\sigma'$ and $t(a)=\sigma$ exactly when $J'
\subsetneq J$.  It is clear that the sets $V(X)$ and $E(X)$ satisfy the properties of a scwol.  Moreover, if $Y$ is a subcomplex of $X$, then the sets $V(Y)$ and $E(Y)$ also satisfy Definition~\ref{d:scwol} above.  In particular, the set of vertices and edges of the standard chamber $K \subset X$ satisfy this definition.  We will denote the scwol associated to $X$ by $\cX$ and that associated to $K$ by $\cK$.  The Kac--Moody group $G$ then acts on the scwol $\cX$ with quotient $\cK$.

\begin{defn}
A \emph{complex of groups} $G(\cY)=(G_\sigma, \psi_a, g_{a,b})$ over a scwol
$\cY$ is given by: \begin{enumerate} \item a group $G_\sigma$ for each
$\sigma \in V(\cY)$, called the \emph{local group} at $\sigma$;
\item a monomorphism $\psi_a: G_{i(a)}\rightarrow G_{t(a)}$ along the edge $a$ for each
$a \in E(\cY)$; and
\item for each pair of composable edges, a twisting element $g_{a,b} \in
G_{t(a)}$, such that \[ \ad(g_{a,b})\circ\psi_{ab} = \psi_a
\circ\psi_b
\] where $\ad(g_{a,b})$ is conjugation by $g_{a,b}$ in $G_{t(a)}$,
and for each triple of composable edges $a,b,c$ the following
\emph{cocycle condition} holds
\[\psi_a(g_{b,c})\,g_{a,bc} = g_{a,b}\,g_{ab,c}.\] \end{enumerate}
\end{defn}

\noindent A complex of groups is \emph{simple} if each $g_{a,b}$ is trivial, and is \emph{trivial} if each local group is trivial.  (A trivial complex of groups must be simple.) 

For example, we construct a simple complex of groups $G(\cK)$ over $\cK$ as follows.  Let $\s \in V(\cK)$.  Then $\s = \sigma_J$ for a unique spherical subset $J \subseteq I$, and we put $G_\s = P_J$.  All monomorphisms along edges of $\cK$ are the natural inclusions, and all $g_{a,b}$ are trivial.  

This complex of groups $G(\cK)$ is canonically induced by the action of $G$ on $\cX$.  For the general construction of a complex of groups induced by a group action on a scwol, see \cite[Section 2.9]{BH}.  A complex of groups is \emph{developable} if it is isomorphic to a complex of groups induced by a group action on a scwol.  Complexes of groups, unlike graphs of groups, are not in general developable.  

We refer the reader to \cite{BH} for the definition of the \emph{fundamental group} $\pi_1(G(\cY))$ and \emph{universal cover} of a (developable) complex of groups $G(\cY)$. The universal cover is a connected, simply-connected scwol, equipped with an action of $\pi_1(G(\cY))$, so that the complex of groups induced by the action of the fundamental group on the universal cover is isomorphic to $G(\cY)$.    For example:
\begin{enumerate}
\item The complex of groups 
 $G(\cK)$ has fundamental group $G$ and universal cover $\cX$.   
 \item If $G(\cY)$ is a trivial complex of groups, then $\pi_1(G(\cY))$ is the fundamental group of the underlying scwol $\cY$.
\end{enumerate}

We finally define morphisms and coverings of complexes of groups.  In the following definitions, $\cY$ and $\cZ$ are scwols, and $G(\cY)=(G_\s,\psi_a, g_{a,b})$ and $H(\cZ)=(H_\tau,\theta_a, h_{a,b})$ are complexes of groups over $\cY$ and $\cZ$ respectively.

\begin{defn}\label{d:morphism} Let $f: \cY\to \cZ$ be a (possibly degenerate) morphism of scwols.  A \emph{morphism} $\Phi: G(\cY) \to H(\cZ)$ over $f$ consists of: \begin{enumerate}
\item a homomorphism $\phi_\sigma: G_\sigma \to H_{f(\sigma)}$ for each $\sigma \in V(\cY)$, called
the \emph{local map} at $\s$; and
\item\label{i:commuting} an element $\phi(a) \in H_{t(f(a))}$ for each $a \in E(Y)$, such that the following diagram commutes
\[\xymatrix{
G_{i(a)}   \ar[d]^-{\phi_{i(a)}} \ar[rrr]^{\psi_a} & & & G_{t(a)} \ar[d]^-{\phi_{t(a)}}
\\
H_{f(i(a))}  \ar[rrr]^{\ad(\phi(a))\circ \theta_{f(a)}} & & & H_{f(t(a))}
}\]
and for all pairs of
composable edges $(a,b)$ in $E(\cY)$, \[ \phi_{t(a)}(g_{a,b})\phi(ab) = \phi(a) \,\theta_{f(a)}(\phi(b))h_{f(a),f(b)}. \]
\end{enumerate} \end{defn}

\noindent A morphism is \emph{simple} if each element $\phi(a)$ is trivial.  

\begin{defn}\label{d:covering} A morphism $\Phi:G(\cY) \to H(\cZ)$ over a nondegenerate morphism of
scwols $f:\cY\to \cZ$ is a
\emph{covering of complexes of groups} if further: \begin{enumerate}\item each $\phi_\sigma$ is
injective; and \item \label{i:covbijection} for each $\sigma \in V(\cY)$ and $b \in E(\cZ)$ such that
$t(b) = f(\sigma)$, the map on cosets \[  \Phi_{\s/b}:\left(\coprod_{\substack{a \in f^{-1}(b)\\ t(a)=\sigma}} G_\sigma /
\psi_a(G_{i(a)})\right) \to H_{f(\sigma)} / \theta_b(H_{i(b)})\] induced by $g \mapsto
\phi_\sigma(g)\phi(a)$ is a bijection.\end{enumerate}\end{defn}

We will use the following general result on functoriality of coverings (which is implicit in~\cite{BH}, and stated and proved explicitly in~\cite{LT}).

\begin{theorem}\label{t:coverings} Let $G(\cY)$ and $H(\cZ)$ be complexes of groups over scwols $\cY$ and
$\cZ$ respectively.  Suppose there is a covering of complexes of groups $\Phi:G(\cY) \to H(\cZ)$.  Then $G(\cY)$ is
developable if $H(\cZ)$ is developable, and $\Phi$ induces a monomorphism of fundamental groups $\pi_1(G(\cY)) \hookrightarrow \pi_1(H(\cZ))$
and an equivariant isomorphism of universal covers $\widetilde{G(\cY)} \longrightarrow
\widetilde{H(\cZ)}$. \end{theorem}

\begin{corollary}\label{c:coverings}  Let $G(\cK)$ be the canonical complex of groups induced by the action of $G$ on $\cX$.  Suppose $\G(\cY)$ is a complex of groups over a finite scwol $\cY$ such that each local group in $\G(\cY)$ is finite.  Let $\G$ be the fundamental group of $\G(\cY)$.  If there is a covering $\Phi:\G(\cY) \to G(\cK)$, then $\G$ embeds as a cocompact lattice in $G$ so that $\G \bs \cX$ is isomorphic to $\cY$. 
\end{corollary}

\begin{proof} This follows from the characterisation of cocompact lattices in $G$ in the introduction, together with Theorem \ref{t:coverings} above. \end{proof}

\section{Finite subgroups and finite indexes}\label{s:finite}

We continue assumptions and notation from Section \ref{s:preliminaries} above.
In Section \ref{s:subgroups} we establish useful results about certain finite subgroups of $G$ and their actions upon $X$, which we will rely upon in our constructions of lattices in the proof of Theorem \ref{t:RA} below.  Section \ref{s:indexes} then contains results on finite indexes between certain parabolic subgroups of $G$, which we use in the proofs of both Theorem \ref{t:RA} and Theorem \ref{t:FP}.

\subsection{Finite subgroups}\label{s:subgroups}

We first recall the Levi decomposition (see \cite{CR}).  For each subset $J \subsetneq I$, the standard parabolic subgroup $P_J:=\coprod_{w \in W_J}BwB$ of $G$ has a Levi decomposition, meaning that $$P_J \cong L_J \ltimes U_J$$ where $L_J$ is finite if and only if $J$ is spherical, and $U_J$ is pro--$p$.  In particular: 
\begin{itemize} 
\item If $J = \emptyset$, then putting $B = P_\emptyset$ we have $B \cong T \ltimes U$ where $T$ is a fixed maximal split torus of $G$, hence $|T|$ divides $(q-1)^{n}$, and $U$ is pro--$p$.  
\item If $J = \{ i \}$ for $i \in I$ we write $L_i = L_{\{ i \}}$.  The group $L_i$ factors as $L_i = M_i T$ where $M_i = [L_i,L_i]$ is normalised by $T$.  
In particular, for each $i \in I$, $M_i \cong A_1(q)$, which is isomorphic to either $\SL_2(q)$ or $\PSL_2(q)$, and $L_i/Z(L_i)$ is isomorphic to either  $\PSL_2(q)$ or $\PGL_2(q)$.  
\end{itemize}

We now consider the action of certain finite subgroups of $L_i$ and $M_i$ on the Davis realisation $X$.  We denote the set of chambers of $X$ which contain the vertex $\s_i$ of $K$ by $\cC_i$.  In other words, $\cC_i$ is the set of chambers in the residue of type $i$ containing the standard chamber $K$.  The set $\cC_i$, by construction of the building $\Delta$, is identified with the cosets $P_i/B$.   We denote the cyclic group of order $n$ by $C_n$.

\begin{lemma}\label{l:finite}   The group $M_i$ acts transitively on $\cC_i$.  We also have:    
\begin{enumerate}
\item\label{i:finite p=2} If $p = 2$ the group  
$M_i$ admits a subgroup $A_i \cong C_{q + 1}$ which acts transitively on $\cC_i$, so that $A_i \cap B = 1$.
\item\label{i:finite q=3} If $q \equiv 3 \pmod 4$, then if $T_0 \in \Syl_2(T)$, the group $L_i$ admits a subgroup $A_i$ which acts transitively on $\cC_i$ and is such that $A_i \cap B = T_0$.
\item\label{i:finite q=1} If $q \equiv 1 \pmod 4$, put $r = (q+1)/2$.  Then $M_i$ admits a subgroup $A_i \cong C_r$ which acts on $\cC_i$ with $2$ orbits of equal size, so that for all $g \in G$, $A_i \cap B^g = 1$. 
\end{enumerate}
\end{lemma}

\begin{proof}
Since $\cC_i$ is identified with the cosets $P_i/B$, we have that $M_i$ acts transitively on $\cC_i$.

Now assume that $p = 2$.  As explained in \cite[Section 3.2.1]{CT}, in this case the maximal non-split torus $H_i$ of $M_i$ satisfies $H_i \cong C_{q+1}$.  The same proof as for \cite[Lemma 3.2]{CT} then shows that $H_i$ acts simply transitively on $\cC_i$, so in particular $H_i \cap B = 1$.  Put $A_i := H_i$.

For $q$ odd, let $H_i \leq L_i$ be a maximal non-split torus of $M_i$ so that $N_T(H_i)$ is as big as possible.  Suppose that $q \equiv 3 \pmod 4$.  To prove \eqref{i:finite q=3}, take $T_0 \in \Syl_2(T)$ and consider the group $N_{L_i}(H_i)$. If $L_i/Z(L_i)\cong \PSL_2(q)$, then $N_{L_i}(H_i)=N_{M_i}(H_i)C_T(L_i)$ where $C_T(L_i)\cap N_{M_i}(H_i)=Z(M_i)$ and we have $[C_T(L_i),N_{M_i}(H_i)]=1$. In this case put $A_i:=N_{M_i}(H_i)T_0$.   By the same proof as for \cite[Lemma 3.3]{CT}, the group $N_{M_i}(H_i)$ acts transitively on $\cC_i$, and therefore so does $A_i$.  Moreover $A_i \cap B = A_i \cap T=T_0$.  If on the other hand $L_i/Z(L_i)\cong \PGL_2(q)$, then $N_{L_i}(H_i)=H_iQ_i'T_0C_T(L_i)$ where $Q_i'\in \Syl_2(C_{L_i}(H_i))$, $C_T(L_i)\cap H_iQ_i'T_0=C_{T_0}(L_i)$ 
and $[C_T(L_i),H_iQ_i'T_0]=1$.  Notice that a different way to describe $N_{L_i}(H_i)$ is $N_{L_i}(H_i) = N_{M_i}(H_i)T_0C_T(L_i)$, that is, $H_iQ_i'T_0 = N_{M_i}(H_i)T_0$.  Thus again take $A_i:=N_{M_i}(H_i)T_0$.  By the same proof as for \cite[Lemma 3.4]{CT}, it follows that the group $A_i$ acts transitively on $\cC_i$.  Moreover $A_i \cap B = A_i \cap T=T_0$.

Now suppose that $q \equiv 1 \pmod 4$.  Then $r = (q+1)/2$ is odd and $H_i\cong C_2\times C_r$.  Take $A_i\leq H_i$ such that $A_i\cong C_r$. That is, $A_i$ is the unique subgroup of $H_i$ of index $2$.  By the same arguments as in \cite[Section 3.3.2]{CT}, each $A_i$ has $2$ orbits of equal size $(q+1)/2$ on $\cC_i$.  Let $g \in G$.  Since $(|A_i|,|T|) = 1$, we have $A_i \cap T^g = 1$.  Thus $A_i \cap B^g = 1$, as required.\end{proof}

In our next results, we consider the groups generated by certain pairs of finite subgroups of $L_i$ or $M_i$. If $J = \{i,j\}$ is a spherical subset with $|J| = 2$, we denote the set of chambers of $X$ which contain the vertex  $\s_{ij} := \s_{\{i,j\}}$ of $K$ by $\cC_{ij}$.   That is, $\cC_{ij}$ is the set of chambers in the residue of type $J=\{i,j\}$ which contain the standard chamber $K$.

\begin{lemma}\label{l:M_J 2}   Let $J = \{ i,j\}$ be a subset of $I$ with $|J| = 2$.  Define $M_J:=\la M_i, M_j\ra$.   If $m_{ij} = 2$, then $[M_i,M_j] = 1$ and $M_J=M_i\circ M_j$.
\end{lemma}

\begin{proof}  We use the presentation of incomplete Kac--Moody groups over $\F_q$, which was introduced by Tits and first stated explicitly by Carter (cf. \cite{CT}).  From this presentation, for each $k\in I$, $M_k=\langle x_k(t_k), x_{-k}(s_k) \mid t_k,s_k\in\F_q\rangle$, and in the case that $m_{ij}=2$ we have $$[x_i(t_i), x_j(t_j)]=[x_{-i}(s_i), x_{-j}(s_j)]=[x_i(t_i), x_{-j}(s_j)]=[x_{-i}(s_i), x_j(t_j)]= 1$$ for $t_i,s_i,t_j,s_j\in\F_q$. The result now follows immediately. 
\end{proof}

\begin{lemma}\label{l:i:2}  Let $J = \{i,j\}$ be a subset of $I$ with $|J| = 2$.  Suppose that $m_{ij} = 2$.  
\begin{enumerate}
\item\label{i:A_i commute} For all $A_i \leq M_i$ and $A_j \leq M_j$ we have $\la A_i,A_j \ra = A_i \times A_j$.  \item\label{i:A_i q=3} If $q \equiv 3 \pmod 4$, $T_0 \in \Syl_2(T)$ and $A_i \leq L_i$ and $A_j \leq L_j$ are the groups constructed in the proof of Lemma \ref{l:finite}\eqref{i:finite q=3} above, then  \[ \la A_i, A_j \ra = (N_{M_i}(H_i) \circ N_{M_j}(H_j))T_0.\] \end{enumerate}
\end{lemma}

\begin{proof} 
The result follows from Lemma \ref{l:M_J 2} above, together with the proof of Lemma \ref{l:finite}\eqref{i:finite q=3}.  
\end{proof}

\subsection{Finite indexes}\label{s:indexes}

We will also use the following facts, which concern finite indexes between certain parabolic subgroups of $G$.

\begin{lemma}\label{l:RA index}  Let $J$ be a nonempty subset of $I$.  Suppose that $m_{ij} = 2$ for all $i, j \in J$ with $i \neq j$.  Then all $J' \subset J$, 
$$\left|P_J:P_{J'}\right| = (q+1)^k$$ where $k = |J \setminus J'|$.  In particular,  $|P_i:B| = (q+1)$ for all $i \in I$.
\end{lemma}

\begin{proof}  
Let $J$ be a non-empty subset as defined above. Then $P_J=U_J\rtimes L_J$ with $L_J=M_JT$. Moreover, Lemma~\ref{l:M_J 2} implies that $M_J=\circ_{i\in J} M_i$. 
Now, $U_J\leq B\leq P_J$, and so $|P_J:B|=|\overline{P_J}:\overline{B}|$ where $\overline{P}_J:=P_J/U_J\cong L_J$ and $\overline{B}:=B/U_J$. Since $|\overline{B}|=q^{|J|}|T|$ while $|L_J|=(q+1)^{|J|}q^{|J|}|T|$, it follows that
$|P_J:B|=(q+1)^{|J|}$.

Take $J'\subseteq J$. Then an identical proof gives us that $|P_{J'}:B|=(q+1)^{|J'|}$.
Therefore $|P_J:P_{J'}|=\frac{|P_J:B|}{|P_{J'}:B|}=(q+1)^{|J|-|J'|}=(q+1)^{|J\setminus J'|}=(q+1)^k$.
\end{proof}

\begin{lemma}\label{l:indexes finite}  Let $J$ be a nonempty subset of $I$.   If $J$ is spherical, the index $|P_J:B|$ is finite.  Hence for all $J' \subsetneq J$ the index $|P_J:P_{J'}|$ is finite.  
\end{lemma}
\begin{proof}
Assume that $J$ is a spherical subset of $I$. Then $P_J=U_J\rtimes L_J$ and  $L_J=M_JT$ where $M_J$ is a semisimple group of Lie type normalised by $T$ and  $T$ is a homomorphic image of $C_{q-1}^{|J|}$. In particular, $L_J$ is a finite group.
Now, $U_J\leq B\leq P_J$ and as $P_J/U_J\cong L_J$,  $|P_J:B|\leq |P_J:U_J|=|L_J|$ is finite. The fact that $B\subseteq P_{J'}\subseteq P_J$ finishes the proof.
\end{proof}

\section{Proof of Theorem \ref{t:RA}}\label{s:proof RA}

In this section we prove Theorem \ref{t:RA} and the Surface Subgroup Corollary.   We construct chamber-transitive lattices and so prove part \eqref{i:p=2, q=3, RA} of Theorem \ref{t:RA} in Section \ref{s:proof RA chamber-transitive}, and then prove parts \eqref{i:free} and \eqref{i:Ipq} in Sections \ref{s:proof RA free} and \ref{s:proof RA Ipq} respectively.  In Section \ref{s:discuss q=1} we discuss why the case $q \equiv 1 \pmod 4$ is more difficult.  Section~\ref{s:surface} establishes the Surface Subgroup Corollary.

\subsection{Proof of (1) of Theorem \ref{t:RA}}\label{s:proof RA chamber-transitive}

We construct chamber-transitive lattices in $G$ as in Theorem \ref{t:RA} using the following embedding criterion.  We will apply this criterion in the case $p = 2$ in Section~\ref{s:p=2} and then in the case $q \equiv 3 \pmod 4$ in Section~\ref{s:q=3}.

\begin{prop}\label{p:covering RA}  Let $G$ be a complete Kac--Moody group of rank $n \geq 3$ with generalised Cartan matrix $A = (a_{ij})$, defined over the finite field $\F_q$ where $q = p^h$ and $p$ is prime.
Assume that for all $i \neq j$, either $a_{ij}a_{ji} = 0$ or $a_{ij}a_{ji} \geq 4$, and that if $a_{ij}a_{ji} \geq 4$ then $|a_{ij}|, |a_{ji}| \geq 2$.
 
For all $i \in I$, fix a finite group $A_i \leq P_i$.  For each nonempty $J \in \cS$, define $$A_J :=\la A_j \mid j \in J\ra \leq P_J,$$ so that $A_{\{ i\}} = A_i$.  By abuse of notation put $A_\emptyset = A_0 := \cap_{i \in I} A_i \leq B$.  Suppose that for all $J \in \cS$:
\begin{enumerate}
\item\label{i:indexes} for all $j \in J$, we have $|A_J: A_{J - \{ j \}}| = q+1$; and
\item\label{i:intersection RA} for all proper subsets $J' \subsetneq J$, we have $A_J \cap P_{J'} = A_{J'}$.
\end{enumerate}
Let $\G(\cK)$ be the simple complex of groups over $\cK$ with the local group at $\sigma_J$ being $A_J$, for all $J \in \cS$, and all monomorphisms inclusions.  
Then the fundamental group of $\G(\cK)$ embeds as a chamber-transitive cocompact lattice in $G$.
\end{prop}

\begin{proof}  We will construct a simple covering of complexes of groups $\Phi:\G(\cK) \to G(\cK)$ over the identity map $f:\cK \to \cK$.  The conclusion then follows from Corollary \ref{c:coverings} above.

We define each local map $\phi_\sigma:A_J \to P_J$ to be inclusion.  Then each $\phi_\sigma$ is injective.  Let $\s \in V(\cK)$ and $b \in E(\cK)$ be such that $t(b) = f(\sigma) = \sigma$.  Then since each $\phi(a) = 1$, the map $\Phi_{\s/b}$ in Definition \ref{d:covering} above is just the map induced by inclusion 
\[ \Phi_{\s/b}:A_J/A_{J'} \to P_J/P_{J'} \]
where $i(b)$ has type $J'$ and $t(b)$ has type $J$ (with $J' \subsetneq J$).  From Lemma \ref{l:RA index} above and \eqref{i:indexes} in the statement of Proposition \ref{p:covering RA}, it follows that $$|A_J:A_{J'}| = |P_J:P_{J'}| = (q+1)^k < \infty$$ where $k = |J \setminus J'|$.  So to conclude that $\Phi_{\s/b}$ is a bijection, it suffices to show that $\Phi_{\s/b}$ is injective.  This is implied by \eqref{i:intersection RA} in the statement of Proposition \ref{p:covering RA}. 
\end{proof}

We now apply Proposition \ref{p:covering RA} in the cases $p = 2$ and $q \equiv 3 \pmod 4$, and so complete the proof of Theorem \ref{t:RA}\eqref{i:p=2, q=3, RA}.

\subsubsection{Case $p = 2$}\label{s:p=2}  In this case, let each $A_i \leq P_i$ in the statement of Proposition \ref{p:covering RA} above be the cyclic group $C_{q+1} \leq M_i$ from Lemma \ref{l:finite}\eqref{i:finite p=2} above.  Then $A_0 = \cap_{i \in I} A_i \leq B$ is trivial, since $A_i \cap B = 1$ for each $i$.  If $J$ is any nonempty spherical subset then $m_{jj'} = 2$ for all $j,j' \in J$ with $j \neq j'$, so we have $A_J = \prod_{j \in J} A_j$ by Lemma \ref{l:i:2}\eqref{i:A_i commute} above.  It follows that \eqref{i:indexes} and \eqref{i:intersection RA} of Proposition \ref{p:covering RA} above are satisfied.  Hence $G$ admits a chamber-transitive lattice in this case.

\subsubsection{Case $q \equiv 3 \pmod 4$}\label{s:q=3}  In this case, take $T_0 \in \Syl_2(T)$ and let each $A_i \leq P_i$ in the statement of Proposition \ref{p:covering RA} above be the group $A_i \leq L_i$ from Lemma \ref{l:finite}\eqref{i:finite q=3} above.  By the construction of the $A_i$, we obtain that, independently of whether each $L_i/Z(L_i)$ is isomorphic to $\PSL_2(q)$ or $\PGL_2(q)$, $A_i \cap A_j = T_0$ for all $1 \leq i \neq j \leq n$.  Thus $A_0 = \cap_{i \in I} A_i = T_0$.  In particular, since $A_i$ acts transitively on the set $\cC_i$, which has cardinality $(q+1)$, and $A_i \cap B =T_0 $, it follows that $|A_i:A_0| = (q+1)$ for all $i \in I$.

Now fix any $J\in\mathcal{S}$ with $|J| \geq 2$.  Then using Lemma \ref{l:i:2}\eqref{i:A_i q=3} above we obtain that $A_J=(\circ_{j\in J} N_{M_j}(H_j))T_0$.  Hence, for all $j\in J$, we have  
$$|A_J:A_{J-\{ j\}}|=|(\circ_{j'\in J} N_{M_{j'}}(H_{j'}))T_0:(\circ_{j'\in J, j'\neq j} N_{M_{j'}}(H_{j'}))T_0|=|N_{M_j}(H_j)/Z(M_j)|=q+1$$
and so \eqref{i:indexes} of Proposition \ref{p:covering RA} above holds.  Also, for any $J' \subsetneq J$,
$$A_J\cap P_{J'}=(\circ_{j\in J} N_{M_j}(H_j))T_0\cap P_{J'}=(\circ_{j\in J'} N_{M_j}(H_j))T_0=A_{J'}$$
and so \eqref{i:intersection RA} of Proposition~\ref{p:covering RA} above holds.  Hence $G$ admits a chamber-transitive lattice in this case.

We have now completed the proof of part \eqref{i:p=2, q=3, RA} of Theorem \ref{t:RA}.  

\subsection{Proof of (2a) of Theorem \ref{t:RA}}\label{s:proof RA free}

We continue the proof of Theorem \ref{t:RA}.  In this section, we are in the case that $q \equiv 1 \pmod 4$ and $|a_{ij}| \geq 2$ for all $1 \leq i \neq j \leq n$.   We will use Proposition \ref{p:covering RA free} below to construct a cocompact lattice in $G$ which has $2$ orbits of chambers and is transitive on each type of panel.  

We first construct a suitable scwol $\cY$.  Since $m_{ij} = \infty$ for all $i \neq j$, the only nonempty spherical subsets of $S$ are the sets $\{ i \}$ for $i \in I$.   The nerve $L$ thus consists of the $n$ vertices $\sigma_i$ with no higher-dimensional simplices, and the chamber $K$ is the star-graph with central vertex $\sigma_\emptyset$ of valence $n$.  
 Note that the mirror $K_i$ is reduced to the vertex $\sigma_i$.  The Davis geometric realisation $X$ is then a biregular tree of bivalence $(n,q+1)$, with each vertex of degree $n$ the central vertex of a copy of $K$.  
 
 Define a simplicial complex 
\[ Y := \left( \{1,2\} \times K \right) / \sim\]
where for $l,l' \in\{ 1,2\}$ and $x, x' \in K$, $(l,x) \sim (l',x')$ if and only if, for some $i \in I$, $x = x' = \sigma_i$.  That is, $Y$ is obtained by taking two copies of $K$ and gluing them together along their mirrors of the same type.   For $l =1,2$, denote by $\sigma^l_{\emptyset}$ the central vertex of $(l,K) \subset Y$.  We may form a scwol $\cY$ associated to $Y$ by orienting each edge of $(l,K) \subset Y$ to have initial vertex $\sigma^l_\emptyset$.  Denote by $a^l_{\emptyset,i}$ the edge of $\cY$ from $\sigma^l_\emptyset$ to $\sigma_i$.  
As $\cY$ has no composable edges, any complex of groups over $\cY$ will be simple.  

The following result establishes part \eqref{i:free} of Theorem \ref{t:RA}.  

\begin{prop}\label{p:covering RA free} Let $G$ be a complete Kac--Moody group of rank $n \geq 2$ with generalised Cartan matrix $A = (a_{ij})$, defined over the finite field $\F_q$ where $q = p^h$ and $p$ is prime.
Assume that for all $i \neq j$ we have $|a_{ij}| \geq 2$ , and that $q \equiv 1 \pmod 4$.  

For all $i \in I$, let $A_i \leq M_i$ be the group $A_i \cong C_r$ from Lemma~\ref{l:finite}\eqref{i:finite q=1} above, where $r = (q+1)/2$.  Let $\G(\cY)$ be the simple complex of groups over $\cY$ with:
\begin{itemize} 
\item the local group at $\sigma^l_{\emptyset}$ being trivial, for $l \in \{ 1,2 \}$, and the local group at $\sigma_i$ being $A_i$, for all $i \in I$; and 
\item all monomorphisms inclusions (of the trivial group).
\end{itemize}
Then the fundamental group of $\G(\cY)$ embeds as a cocompact lattice in $G$ which has $2$ orbits of chambers and is transitive on each type of panel.
\end{prop}

\begin{proof}   
Let $f:\cY \to \cK$ be the only possible type-preserving morphism.  Then $f$ is clearly nondegenerate.  We will construct a covering of complexes of groups $\Phi: \G(\cY) \to G(\cK)$ over $f$.   (This covering will not be simple.)  The conclusion then follows from Corollary \ref{c:coverings} above.

Since the group $P_i$ acts transitively on $\cC_i$, we may choose $g_i \in P_i$ so that $\cK$ and $g_i\cK$ represent the $2$ orbits of $A_i$ on $\cC_i$.  By Lemma~\ref{l:finite}\eqref{i:finite q=1}, we have $A_i \cap B = A_i \cap B^{g_i} = 1$.

To simplify notation, for $l \in \{1,2\}$ and each $i \in I$, denote the monomorphism along the edge $a^l_{\emptyset,i}$ by $\psi^l_{\emptyset,i}$, and denote by $\phi^l_\emptyset$ the local map $\G_{\sigma^l_{\emptyset}} \to G_{\sigma_\emptyset}$.  
Let the local map $\phi^l_\emptyset$ be the inclusion $1 \hookrightarrow B$.   For each $i \in I$, let the local map $\phi_{\sigma_i}:\G_{\sigma_i} \to G_{\s_i}$ be the inclusion $A_i \hookrightarrow P_i$.  Then each local map is injective.  
Put $\phi(a^1_{\emptyset,i}) = 1$ and $\phi(a^2_{\emptyset,i}) = g_i$.  Since the local groups of $\G(\cY)$ at $\s^l_\emptyset$ are trivial for $l = 1,2$, it follows that $\Phi$ is a morphism of complexes of groups.  

For each $i \in I$, let $b_i$ be the edge of $\cK$ from $\s_\emptyset$ to $\sigma_i$.  To show that $\Phi$ is a covering, we must show that for each $i \in I$, the map
\[ \Phi_{\s_i/b_i}: \left(\coprod_{l \in \{1,2\}} A_i/\psi^l_{\emptyset,i}(1) \right) \to P_i/B \]
induced by $g \mapsto \phi_{\s_i}(g)\phi(a^l_{\emptyset,i}) = g\phi(a^l_{\emptyset,i})$ is a bijection.  Since $2|A_i| = (q+1) = |P_i:B|$, it suffices to show that $\Phi_{\s_i/b_i}$ is injective.  For this, we first note that since the chambers $\cK$ and $g_i \cK$ represent pairwise distinct $A_i$--orbits on $\cC_i$, for all $g, h \in A_i$ the cosets $gB = g\phi(a^1_{\emptyset,i})B$ and $hg_iB =  h\phi(a^2_{\emptyset,i})B$ are distinct.  Then from $A_i \cap B = 1$ and $A_i \cap B^{g_i} = 1$, it follows that $\Phi_{\s_i/b_i}$ is injective.

Thus we have constructed a covering of complexes of groups $\Phi:\G(\cY) \to G(\cK)$.
\end{proof}

This completes the proof of part \eqref{i:free} of Theorem \ref{t:RA}.

\subsection{Proof of (2b) of Theorem \ref{t:RA}}\label{s:proof RA Ipq}

We now complete the proof of Theorem \ref{t:RA}, by considering the case that $q \equiv 1 \pmod 4$, and that for $n \geq 5$ the Weyl group for $G$ has the form
\[W = W_n := \la s_1, \ldots, s_n \mid s_i^2 = 1, \quad (s_i s_{i+1})^2 = 1, \quad \forall i \in \Z/n\Z \ra.\]
We will construct, in Proposition \ref{p:covering RA Ipq} below, a cocompact lattice $\G$ in $G$ which acts with $F$ orbits of chambers, for $F$ any positive integer multiple of $8$.

For $W = W_n$ as above, the nonempty spherical subsets are the sets $\{ i \}$ and $\{i,i+1\}$ (taking $i \in \Z/n\Z$).  Thus the nerve $L$ is an $n$--cycle with vertices the $\sigma_i$ in cyclic order, and the chamber $K$ is the barycentric subdivision of an $n$--gon, with each mirror $K_i$  the barycentric subdivision of an edge of the $n$--gon.  It is then natural to consider $K$ to be the $n$--gon itself rather than its barycentric subdivision, so that each mirror $K_i$ a side of the $n$--gon (numbered cyclically).  If we now metrise $K$ so that it is a right-angled hyperbolic $n$--gon, then the Davis realisation $X$ is a right-angled Fuchsian building with each chamber isometric to $K$, and the link at each vertex the complete bipartite graph $K_{q+1,q+1}$.  This building is known as Bourdon's building \cite{Bo}.

Now let $Y$ be some tessellation of a compact orientable hyperbolic surface by copies of the right-angled hyperbolic $n$--gon $K$.  In \cite{FT}, Futer and Thomas constructed various lattices in the automorphism group of Bourdon's building as fundamental groups of simple complexes of groups over $Y$.  Denote by $Y'$ the first barycentric subdivision of $Y$, with vertex set $V(Y')$ and edge set $E(Y')$.  Each $a \in E(Y')$ corresponds to cells $\tau \subset \s$ of $Y$, and so may be oriented from $i(a) = \s$ to $t(a) = \tau$.  Two edges $a$ and $b$ of $Y'$ are \emph{composable} if $i(a) = t(b)$, in which case there exists an edge $c = ab$ of $Y'$ such that $i(c) = i(b)$, $t(c) = t(a)$ and $a$, $b$ and $c$ form the boundary of a triangle in $Y'$.  (Note that the sets $V(Y')$ and $E(Y')$ satisfy the definition of scwol.)  A \emph{complex of groups over $Y$} assigns a local group $G_\s$ to each $\s \in V(Y')$, and defines monomorphisms $\psi_a: G_\s \to G_\tau$ whenever $a \in E(Y')$ with $i(a) = \s$ and $t(a) = \tau$.  We will call the local groups \emph{face}, \emph{edge} and \emph{vertex} groups as appropriate.  See \cite[Section 2.3]{FT} for more details concerning complexes of groups over polygonal complexes such as $Y$.

Now assume that $F \geq 8$ is a multiple of $8$ and let $g \geq 2$ be such that $F = \frac{8(g-1)}{n-4}$.  Let $Y$ be the tessellation of a closed surface $S_g$ of genus $g$ by $F$ copies of $K$ given in the proof of \cite[Proposition 6.1]{FT}, where this construction is carried out to ensure that all gluings are type-preserving.  Since $K$ is right-angled, at each vertex of $Y$ two (local) geodesics intersect.  In fact, since the gluings are type-preserving, there are no self-intersecting geodesics, each geodesic has a well-defined type $i \in \Z/n\Z$ and two geodesics of types $i, j \in \Z/n\Z$ intersect only if $j = i \pm 1 \pmod n$.  If $h_1, \ldots, h_N$ are the geodesics of this tiling, then by \cite[Proposition 6.1]{FT}, the geodesics $h_k$ may be oriented so that $\sum [h_k] = 0 \in H_1(S_g)$.  We fix this orientation on the geodesics $h_1,\ldots,h_N$.

\begin{prop}\label{p:covering RA Ipq} Let $G$ be a complete Kac--Moody group of rank $n \geq 2$ with generalised Cartan matrix $A = (a_{ij})$, defined over the finite field $\F_q$ where $q = p^h$ and $p$ is prime.
Assume that if $a_{ij}a_{ji} \geq 4$ then $|a_{ij}|, |a_{ji}| \geq 2$, that the Weyl group $W$ of $G$ satisfies  $W = W_n$ for $n \geq 5$ and that $q \equiv 1 \pmod 4$.

For each $i \in \Z/n\Z$, let $A_i \leq M_i$ be the group $A_i \cong C_r$ from Lemma~\ref{l:finite}\eqref{i:finite q=1} above, where $r = (q+1)/2$.  Let $\G(Y)$ be the simple complex of groups over $Y$ with:
\begin{itemize}
\item all face groups trivial;
\item all edge groups for edges contained in a geodesic of type $i$ the group $A_i$; 
\item all vertex groups at intersections of geodesics of type $i$ and type $j$ the group $A_i \times A_j$; and
\item all monomorphisms the natural inclusions.
\end{itemize}
Then the fundamental group of $\G(Y)$ embeds as a cocompact lattice in $G$ which has $F$ orbits of chambers.
\end{prop}

\begin{proof}  As in previous results, we construct a covering of complexes of groups $\Phi$, and the conclusion then follows from Corollary \ref{c:coverings} above.  By a slight abuse of notation, this covering will be from $\G(Y)$ to $G(\cK)$.

By Lemma \ref{l:i:2}\eqref{i:A_i commute} above, for each $i \in \Z/n\Z$ and $j = i \pm 1 \pmod n$, we have $\la A_i, A_j \ra = A_i \times A_j$.  We may thus define the local maps in $\Phi$ to be the inclusions $1 \hookrightarrow B$ of face groups, $A_i \hookrightarrow P_i$ of edge groups and $A_i \times A_j \hookrightarrow P_{ij}$ of vertex groups.  

The elements $\phi(a)$ for $a \in E(Y')$ are defined as follows.  Since the group $M_i$ acts transitively on $\cC_i$, we may choose $g_i \in M_i$ so that $\cK$ and $g_i\cK$ represent the $2$ orbits of $A_i$ on $\cC_i$.  By Lemma~\ref{l:finite}\eqref{i:finite q=1}, we have $A_i \cap B = A_i \cap B^{g_i} = 1$.  Now for each (oriented) geodesic $h$ of type $i$ in $Y$, consider the set of edges $a \in E(Y')$ such that $t(a)$ is in $h$ and $a$ is not the product of two composable edges.  If $i(a)$ is to the left of $h$, put $\phi(a) = 1$ and if $i(a)$ is to the right of $h$, put $\phi(a) = g_i\in M_i$.  Then for each pair of composable edges $(a,b)$, we define $\phi(ab) = \phi(a)\phi(b)$.  This is well-defined by our previous choices of $\phi(a)$ and the fact that $[M_i,M_{i+1}] = 1$ from Lemma \ref{l:M_J 2} above.  It is then easy to check that $\Phi$ is a morphism of complexes of groups.

For each $i \in I$, let $b_i$ be the edge of $\cK$ from $\s_\emptyset$ to $\sigma_i$.  The proof that $\Phi_{\s_i/b_i}$ is a bijection is the same as in Proposition \ref{p:covering RA free} above.  To complete the proof that $\Phi$ is a covering, fix $\s \in f^{-1}(\s_{ij})$, where $j = i \pm 1 \pmod n$.   

Let $b$ be the edge of $\cK$ from $\s_i$ to $\s_{ij}$ and consider
\[ \Phi_{\s/b}: \left(\coprod_{l =1,2} (A_i \times A_j)/\psi_{a_l}(A_i) \right) \to P_{ij}/P_i \]
which is induced by $g \mapsto \phi_{\s_{ij}}(g)\phi(a_l) = g\phi(a_l)$, where $a_1$ and $a_2$ are the two edges of $E(Y')$ with initial vertex of type $i$ and terminal vertex $\s$.  Without loss of generality $\phi(a_1) = 1$ and $\phi(a_2) = g_j$.  By our choice of $g_j$,  for all $g \in A_i \times A_j$ the cosets $gP_i$ and $gg_jP_i$ are distinct.  It now suffices to show that the elements of $A_j$ and of $A_jg_j$ form a transversal for $P_{ij}/P_i$.  This is immediate from the observation that the group $A_j$ acts fixed-point-freely on $P_{ij}/P_i$, with $2$ orbits, one represented by $P_i$ and one by $g_jP_i$. 

Now let $b_\emptyset$ be the edge of $\cK$ from $\s_\emptyset$ to $\s_{ij}$ and consider
\[ \Phi_{\s/b_\emptyset}: \left(\coprod_{l =1}^4 (A_i \times A_j)/\psi_{c_l}(1) \right) \to P_{ij}/B \]
which is induced by $g \mapsto \phi_{\s_{ij}}(g)\phi(c_l) = g\phi(c_l)$, where $c_1, c_2, c_3, c_4$ are the $4$ edges of $E(Y')$ with initial vertex of type $\emptyset$ and terminal vertex $\s$.  Then without loss of generality $\phi(c_1) = 1$, $\phi(c_2) = g_i$, $\phi(c_3) = g_j$ and $\phi(c_4) = g_ig_j$.  The proof that $\Phi_{\s_{ij},b_\emptyset}$ is a bijection follows from the observation that a transversal for $P_{ij}/B$ may be obtained from the transversals for $P_i/B$ and $P_{ij}/P_i$ in previous paragraphs.

Thus we have constructed a covering of complexes of groups $\Phi:\G(\cY) \to G(\cK)$.
\end{proof}

This completes the proof of Theorem \ref{t:RA}.

\subsection{Discussion of case $q \equiv 1 \pmod 4$}\label{s:discuss q=1}

We now briefly discuss the reasons why, when $q \equiv 1 \pmod 4$, we do not know whether all $G$ as in Theorem \ref{t:RA} admit a cocompact lattice.  Our discussion here is mostly independent of whether the Weyl group of $G$ is right-angled.

In the affine case $G = \SL_n(\F_q((t)))$, where the Weyl group is of type $\tilde{A}_{n-1}$, a lattice $\G < G$ is cocompact if and only if it does not contain any $p$--elements.  See \cite[Proposition 24]{CRT} for a proof of this analogue of Godement's Cocompactness Criterion.  If $G$ is a complete Kac--Moody group of rank $2$ over $\F_q$, we conjectured in \cite{CT} that a lattice $\G < G$ is cocompact if and only if it does not contain $p$--elements. 

For $G$ of higher rank, it seems reasonable to make the same conjecture, and thus to concentrate the search for cocompact lattices on groups which do not contain $p$--elements.  Now, by the same arguments as in \cite[Section 5]{CT}, if $q \equiv 1 \pmod 4$ there is no finite group $A_i \leq P_i$ of order coprime to $p$ such that $A_i$ acts transitively on $P_i/B$.  Hence when $q \equiv 1 \pmod 4$, we cannot hope to construct a chamber-transitive cocompact lattice using stabilisers of order coprime to $p$.  Thus in the quotient, each panel must be contained in at least $2$ distinct chambers.   In Sections \ref{s:proof RA free} and \ref{s:proof RA Ipq} above, where we have obtained a cocompact lattice in some cases with $q \equiv 1 \pmod 4$, we were able to construct a ``reasonable" finite quotient space in which each panel is contained in exactly $2$ chambers.  We do not know how to construct a suitable finite quotient space for general right-angled $W$ when $q \equiv 1 \pmod 4$.

\subsection{Proof of the Surface Subgroup Corollary}\label{s:surface}

In this section we prove the Surface Subgroup Corollary, which was stated in the introduction.  The case $p=2$ is established in Proposition \ref{p:p=2 surface} below, and the remainder of the Corollary in Proposition \ref{p:q=1 surface}. 

We first note the following.

\begin{lemma}  Let $G$ be a complete Kac--Moody group of rank $n \geq 3$, with Weyl group $W$.  If $W$ is word-hyperbolic then any cocompact lattice in $G$ is word-hyperbolic.
\end{lemma}

\begin{proof}  This follows from \cite[Corollary 18.3.10]{D} and the fact that cocompact lattices in $G$ act properly discontinuously and cocompactly on the Davis realisation $X$.
\end{proof}

We also note that when $p = 2$ our Surface Subgroup Corollary applies to some cases in which $W$ is \emph{not} word-hyperbolic (see Moussong's Theorem, \cite[Corollary 12.6.3]{D}).  

\begin{prop}\label{p:p=2 surface}  Let $G$ be as in Proposition \ref{p:covering RA free}.  Assume that $p = 2$ and that $W$ has a special subgroup isomorphic to $W_{n'}$ for some $n' \geq 5$.  Let $\G$ be the cocompact lattice in $G$ constructed in Section \ref{s:p=2} above.  Then $\G$ has a surface subgroup.
\end{prop}

\begin{proof}  By construction and the definition of the fundamental group of a complex of groups, the lattice $\G$ has presentation with generating set the finite cyclic groups $A_i$ for $1 \leq i \leq n$, and relations the relations in the $A_i$ together with  $[A_i,A_j] = 1$ whenever $m_{ij} = 2$.  In other words, the lattice $\G$ is the graph product of the groups $A_i$ over the graph which is the $1$--skeleton of the nerve~$L$.

By assumption, $W$ has a special subgroup isomorphic to $W_{n'}$ for some $n' \geq 5$, and so the $1$--skeleton of $L$ contains a full subgraph which is a cycle on $n'$ vertices, say $s_{i_1},\ldots,s_{i_{n'}}$, in cyclic order.  Let $\G'$ be the graph product of the groups $A_{i_1},\ldots,A_{i_{n'}}$ over this $n'$--cycle.  Then $\G'$ embeds in $\G$ (see Green \cite{G}).  Thus it suffices to show that $\G'$ has a surface subgroup.

Since $n' \geq 5$ and $A_{i_1}, \ldots,A_{i_{n'}}$ are nontrivial, the fact that $\G'$ has a surface subgroup follows from either Holt--Rees \cite[Theorem 3.1]{HR} or, independently, Kim \cite[Corollary 6]{K}.
\end{proof}

\begin{prop}\label{p:q=1 surface}  Let $G$ be as Proposition \ref{p:covering RA Ipq} above, with $W = W_n$ for $n \geq 5$.
\begin{enumerate}
\item If $q \equiv 1 \pmod 4$ then the cocompact lattice $\G$ constructed in Section \ref{s:proof RA Ipq} above has a surface subgroup.
\item If $n \geq 6$ then for all $q$, every cocompact lattice in $G$ has a surface subgroup.
\end{enumerate}
\end{prop}

\begin{proof}  The construction of $\G$ in Section \ref{s:proof RA Ipq} uses a simple complex of groups from Futer--Thomas \cite{FT}, so that by the same arguments as in \cite{FT}, the group $\G$ has a surface subgroup.   

Now suppose $n \geq 6$ and that $q$ is any power of a prime.  Any cocompact lattice $\G < G$ is also a cocompact lattice in $\Aut(X)$, by the characterisation of cocompact lattices stated in the introduction.  By \cite[Corollary 1.3]{FT}, every cocompact lattice in $\Aut(X)$ has a surface subgroup.  The result follows.
\end{proof}

\section{Proof of Theorem \ref{t:FP}}\label{s:proof FP}

We now prove Theorem \ref{t:FP}.  
Throughout this section $G$ is a complete Kac--Moody group of rank $n \geq 2$ with generalised Cartan matrix $A = (a_{ij})$, defined over the finite field $\F_q$ where $q = p^h$ and $p$ is prime, such that for all $i \neq j$,  if $a_{ij}a_{ji} \geq 4$ then $|a_{ij}|, |a_{ji}| \geq 2$.  

We begin with an elementary group-theoretic observation.

\begin{lemma}\label{l:transversals}  Suppose $U \leq V \leq H$ are groups.  Choose transversals $\{ c_i \}$ for $H/U$ and $\{ a_j \}$ for $H/V$.  Then for each $j$, the set
\[ B_j := \{ a_j^{-1} c_i \mid c_i U \subset a_jV \} \]
forms a transversal for $V/U$.
\end{lemma}

\begin{proof}  Fix $j$.  If $c_iU \subset a_jV$ it is clear that $a_j^{-1}c_i \in V$.  Now suppose that $a_j^{-1}c_i U = a_j^{-1}c_{i'}U$, where $i$ and $i'$ are such that $c_iU$ and $c_{i'}U$ are both contained in $a_jV$.  Then $c_iU = c_{i'}U$, but since $V/U \subset H/U$, this implies that $i = i'$ and hence $a_j^{-1}c_i = a_j^{-1}c_{i'}$.   To complete the proof, we must show that for all $v \in V$, there is an $i$ so that $vU = a_j^{-1}c_iU$ and $c_iU \subset a_jV$.  For this, consider the coset $a_jvU \in H/U$.  Then $a_jvU = c_iU$ for some $i$, by definition of the $c_i$.  It follows that $c_iU = a_jvU \subset a_jV$, as required.  This completes the proof.
\end{proof}

We now assume that $I$ can be partitioned into nonempty spherical subsets $J_1, \ldots, J_N$ ($N \geq 2$), so that for all $j \in J_k$ and $j' \in J_{k'}$ with $k \neq k'$, we have $m_{jj'} = \infty$.  Then as in the statement of Theorem~\ref{t:FP} in the introduction, the Weyl group $W$ 
has a free product decomposition $$W = W_1 * W_2 * \cdots * W_N$$
with each special subgroup $W_k:=W_{J_k}$ spherical.  Hence the nerve $L$ is the disjoint union of simplices $L_1, \ldots, L_N$, where each $L_k$ is the simplex with vertex set $J_k$.  Denote by $K_1,\ldots,K_N$ respectively the cones on the barycentric subdivisions of $L_1, \ldots,L_N$.  Then the  chamber $K$ may obtained by gluing together the cone points of $K_1, \ldots, K_N$.

We next construct a suitable scwol $\cY$.  For this, for each $1 \leq k \leq N$, let $Z_{k}$ be the subcomplex of~$X$ which is the union of all (closed) simplices which contain the vertex $\s_{J_k}$ of the chamber~$K$.  In other words, $Z_{k}$ is the union of the images of the cone $K_k$ in each of the chambers of $X$ which contain $\s_{J_k}$.  Note that by Lemma \ref{l:indexes finite}, for each $1 \leq k \leq N$ the index $$m_k:=|P_{J_k}:B|$$ is finite.  Hence each $Z_{k}$ is a finite simplicial complex containing $m_k$ vertices of type $\emptyset$.  Define $$M := \prod_{k=1}^N m_k < \infty$$ and for $1 \leq k \leq N$ let $$M_k := M/m_k = m_1 m_2 \cdots m_{k-1}m_{k+1} \cdots m_N.$$  For each $1 \leq k \leq N$, let $Y_k$ be the simplicial complex consisting of $M_k$ disjoint copies of $Z_{k}$.  Then each $Y_k$ contains $M$ vertices of type~$\emptyset$.  

Now define $Y$ be the simplicial complex obtained by taking the disjoint union of the $Y_k$, for $1 \leq k \leq N$, and then gluing together collections of $N$ distinct vertices of type $\emptyset$, so that for each $1 \leq k \leq N$, each vertex in $Y$ of type $\emptyset$ is contained in exactly one copy of $Z_{k}$.  By construction, the resulting simplicial complex $Y$ contains $M$ chambers.   Note also that~$Y$ deformation retracts to a finite graph.  

\begin{example}
Suppose $N = 2$.  Then $Y_1$ consists of $m_2$ copies of $Z_{1}$ and $Y_2$ consists of $m_1$ copies of $Z_{2}$.  In $Y$, each of the $M = m_1 m_2$ vertices of type $\emptyset$ is contained in exactly one copy of $Z_1$ and exactly one copy of $Z_2$.  Hence the complex $Y$ deformation retracts to the complete bipartite graph $K_{m_1,m_2}$.  \end{example}

Denote by $\cY$ the scwol associated to $Y$, with edges oriented by inclusion of types.   The following embedding criterion and its corollary establish Theorem \ref{t:FP}.

\begin{prop}\label{p:covering FP}
Let $G$ be as in Theorem \ref{t:FP} and let $\cY$ be the scwol constructed above.  Let $\G(\cY)$ be the trivial complex of groups over $\cY$.  Then the fundamental group of $\G(\cY)$ embeds as a cocompact lattice in $G$, which acts with $M$ orbits of chambers.
\end{prop}

\begin{proof}  We will construct a covering of complexes of groups $\Phi: \G(\cY) \to G(\cK)$ over the only possible type-preserving morphism $f:\cY \to \cK$.  The conclusion then follows from Corollary~\ref{c:coverings} above, together with the observation that $Y$ contains $M$ chambers.  

All local maps in $\Phi$ are inclusion of the trivial group.  
For each $1 \leq k \leq N$, denote by $\cZ_k$ the scwol associated to $Z_{k}$.  We may identify $\cZ_k$ with a subscwol of $\cX$ or of $\cY$, since by construction $\cY$ contains $M_k$ pairwise disjoint embedded copies of $\cZ_k$.  Now for all $k \neq k'$, each copy of $\cZ_k$ in $\cY$ intersects each copy of $\cZ_{k'}$ in at most finitely many vertices of type $\emptyset$.  Therefore, from the definition of a covering, we may define the elements the elements  $\phi(a) \in G_{t(f(a))}$ by restricting $\Phi$ to each copy of $\cZ_k$ in $\cY$.  
\
Fix $1 \leq k \leq N$ and a copy of $\cZ_k$ in $\cY$.  To simplify notation, put $\cZ = \cZ_k$ and $J = J_k$.  In order to define $\Phi$ on $\cZ$, we regard $\cZ$ as a subscwol of $\cX$.  For each $J' \subsetneq J$, define $m_{J'}:=|P_J:P_{J'}|$ and choose a transversal $h_{1,J'},\ldots,h_{m_{J'},J'}$ for $P_J/P_{J'}$.  Then for each edge $a$ of $\cZ$ with $t(a) = \s_{J}$ and $i(a)$ of type $J'$, we may define $\phi(a) = h_{j,J'} \in P_J$ where $i(a) = h_{j,J'}\s_{J'}$.  

Now suppose $b$ is an edge of $\cZ$ such that $t(b) \neq \s_J$.  Then for some $J'' \subsetneq J' \subsetneq J$, the vertex $t(b)$ is of type $J'$ and $i(b)$ is of type $J''$.  Hence $t(b) = h_{j,J'}\s_{J'}$ and $i(b) = h_{j',J''}\s_{J''}$, where $h_{j,J'}$ is an element of the transversal for $P_J/P_{J'}$ and $h_{j',J''}$ is an element of the transversal for $P_J/P_{J''}$.  By construction of the Davis complex, the existence of an edge from $h_{j',J''}\s_{J''}$ to $h_{j,J'}\s_{J'}$ is equivalent to the containment of cosets $h_{j',J''}P_{J''} \subset h_{j,J'}P_{J'}$.  Thus $h_{j,J'}^{-1}h_{j',J''} \in P_{J'}$ and we may put $\phi(b) = h_{j,J'}^{-1}h_{j',J''} \in P_{J'}$.

It is now straightforward to verify that $\Phi$ is a morphism of complexes of groups.   To show that $\Phi$ is a covering, it suffices to fix $J = J_k$ as above, and 
 consider the preimages of vertices in $\cK$ of nonempty type $J' \subseteq J$. There are two cases, as follows.
\begin{itemize}
\item
Suppose $\s \in f^{-1}(\s_J)$. For each $J' \subsetneq J$, let $b$ be the unique edge of $\cK$ from $\s_{J'}$ to $\s_J$.  Let $\{a_j\}_{j=1}^{m_{J'}}$ be the set of edges of $\cY$ with terminal vertex $\s$ and initial vertex of type $J'$, indexed so that $\phi(a_j) = h_{j,J'}$.  Then by our choice of the elements $h_{j,J'}$ as a transversal for $P_J/P_{J'}$, the map on cosets 
\[ \Phi_{\s/b}:\left(\coprod_{j = 1}^{m_{J'}} 1/\psi_{a_j}(1)\right) \to P_J/P_{J'} \]
 is immediately a bijection. 
\item
Now suppose that $\s \in f^{-1}(\s_{J'})$ where $J' \subsetneq J$.  We will use Lemma \ref{l:transversals} in this case.  If we regard the copy of $\cZ = \cZ_k$ containing $\s$ as a subcomplex of $\cX$, we have that $\s = h_{j,J'}\s_{J'}$ for a unique $1 \leq j \leq m_{J'}$.  For each $J'' \subsetneq J'$, let $b$ be the unique edge of $\cK$ from $\s_{J''}$ to $\s_J$.  Let $\{b_{j'}\}$ be the set of edges of $\cY$ with terminal vertex $\s$ and initial vertex of type $J''$, for $1 \leq j' \leq |P_{J'}:P_{J''}|$.  Then for each $j'$, there is a $1 \leq j'' \leq m_{J''}$ so that $i(b_{j'}) = h_{j'',J''}\s_{J''}$.  By Lemma \ref{l:transversals} and the remarks above on the construction of the Davis complex, it follows that the set $\{ \phi(b_{j'})\} = \{ h_{j,J'}^{-1}h_{j'',J''} \}$ forms a transversal for $P_{J'}/P_{J''}$.
Hence the map on cosets
\[ \Phi_{\s/b}:\left(\coprod_{j'=1}^{|P_{J'}:P_{J''}|} 1/\psi_{b_{j'}}(1)\right) \to P_{J'}/P_{J''}  \]
is a bijection, as required.   \end{itemize}
Thus we have constructed a covering of complexes of groups $\Phi:\G(\cY) \to G(\cK)$.
\end{proof}

\begin{corollary}  The lattice $\G$ constructed in Proposition \ref{p:covering FP} above is a finitely generated free group.
\end{corollary}

\begin{proof} The lattice $\G$ is the fundamental group of a trivial complex of groups $\G(\cY)$.  Thus $\G$ is isomorphic to the fundamental group of the scwol $\cY$, and hence isomorphic to the (topological) fundamental group of the simplicial complex $Y$.  Since $Y$ deformation retracts onto a finite graph, we conclude that the lattice $\G$ is a finitely generated free group. 
\end{proof}

This completes the proof of Theorem \ref{t:FP}.

\end{document}